\documentclass[hidelinks,onefignum,onetabnum,onealgnum,onethmnum,oneeqnum]{siamart251216}

\usepackage{siunitx}
\protected\def\0{0}
\usepackage{amsfonts}
\usepackage{mathtools}
\usepackage{autonum}
\usepackage{graphicx}
\usepackage{textcomp}
\usepackage{bm}
\usepackage{xcolor}
\usepackage{xspace}
\usepackage[sort]{cite}
\usepackage{pdflscape}
\usepackage{booktabs}
\usepackage{tabularx}
\usepackage{multirow}
\usepackage{xltabular}
\usepackage{longtable}
\usepackage{threeparttable}
\usepackage{rotating}
\usepackage{etoolbox}
\usepackage{float}
\usepackage{algpseudocode}
\usepackage{xurl}
\usepackage{setspace}
\usepackage{placeins}
\usepackage{optidef}
\usepackage{pgfplots}
\pgfplotsset{compat=1.14}
\usepackage{pgfplotstable}
\usepackage{subcaption}
\usepgfplotslibrary{groupplots}
\usetikzlibrary{calc}

\definecolor{lightgray}{gray}{0.9}
\definecolor{darkergray}{gray}{0.8}
\definecolor{forestgreen}{RGB}{34,139,34}


\newcommand{\nnz}{\textrm{nnz}}

\makeatletter
\@cref@capitalisefalse
\crefname{ALG@line}{line}{lines}
\Crefname{ALG@line}{Line}{Lines}
\g@addto@macro\ALG@step{%
  \addtocounter{ALG@line}{-1}%
  \refstepcounter{ALG@line}%
  \expandafter\@cref@getprefix\cref@currentlabel\@nil\cref@currentprefix%
  \xdef\cref@currentprefix{\cref@currentprefix}}%
\g@addto@macro\ALG@beginalgorithmic{%
  \def\cref@currentlabel{%
    [line][\arabic{ALG@line}][\cref@currentprefix]\theALG@line}}%
\makeatother

\crefname{appendix}{the appendix}{appendices}
\crefformat{appendix}{#2the appendix#3}

\pgfplotsset{
    colormap={time_per_it_ratio_colors}{
        rgb255(-0.02)=(122, 35, 0)
        rgb255(0.2279)=(226, 209, 255) 
        rgb255(0.82)=(0, 51, 110)
    },
    colormap={it_ratio_colors}{
        rgb255(0)=(122, 35, 0)
        rgb255(0.5)=(226, 209, 255)
        rgb255(0.999999)=(0, 51, 110)
        rgb255(1)=(0, 150, 40)
        rgb255(1.16)=(0, 230, 0)
    },
    colormap={speedup_colors}{
        rgb255(-0.783)=(0, 51, 110)
        rgb255(-0.3825)=(129, 148, 235)
        rgb255(0)=(230, 230, 255)
        rgb255(0.000001)=(255, 230, 230)
        rgb255(0.3915)=(240, 105, 100)
        rgb255(0.783)=(122, 35, 0)
    }
}

\headers{Multi-Basis Parallel Conjugate Gradient}{Andrew Xu and Shaked Regev}

\ifpdf
\hypersetup{
  pdftitle={Accelerating the Conjugate Gradient Method by Solving Multiple GPU-Parallelized Duplicate Systems},
  pdfauthor={Andrew Xu and Shaked Regev}
}
\fi

\begin{document}

\title{
	Accelerating the Conjugate Gradient Method by Solving Multiple GPU-Parallelized Duplicate Systems
	\thanks{Notice: This manuscript has been authored by UT-Battelle, LLC, under contract DE-AC05-00OR22725 with the US Department of Energy (DOE). The US government retains and the publisher, by accepting the article for publication, acknowledges that the US government retains a nonexclusive, paid-up, irrevocable, worldwide license to publish or reproduce the published form of this manuscript, or allow others to do so, for US government purposes. DOE will provide public access to these results of federally sponsored research in accordance with the DOE Public Access Plan (\href{https://www.energy.gov/doe-public-access-plan}{https://www.energy.gov/doe-public-access-plan}).}
}

\author{
	Andrew Xu\textsuperscript{$\dagger\ddagger\S$} and
	Shaked Regev\textsuperscript{$\dagger\S$}
}
\maketitle
\renewcommand{\thefootnote}{}
\footnotetext{\textsuperscript{$\dagger$}~Oak Ridge National Laboratory, Oak Ridge, TN 37831, USA.}
\footnotetext{\textsuperscript{$\ddagger$}~Duke University, Durham, NC 27707, USA.}
\footnotetext{\textsuperscript{$\S$}~Equal contribution.}
\renewcommand{\thefootnote}{\arabic{footnote}}
\begin{abstract}
	We propose a method to accelerate the conjugate gradient method (\textbf{CG}) on parallel computing hardware by performing block conjugate gradient on one linear system. We aim to reduce the number of iterations until convergence by solving several copies of the system to explore multiple regions of the solution space at once. Although solving multiple systems requires more floating point operations than solving one, these computations are highly parallelizable, even without additional hardware resources. Our method does not replace preconditioning and can be used with any preconditioner. We developed linear regression models to predict iteration reduction and speedup as functions of the coefficient matrix's size, its number of nonzeros, and the solve time with CG. Our experiments show that our method can reduce solve time by up to 6 times in dense systems and up to 5 times in sparse systems that require long solve times with CG relative to their size. 
\end{abstract}

\begin{keywords}
	conjugate gradient, parallel computing, iterative methods
\end{keywords}

\begin{MSCcodes}
	65F10
\end{MSCcodes}

\section{Introduction}
Iterative methods are used to solve large linear systems when direct methods are computationally expensive. This cost is a function of the sparsity structure of the system~\cite{Trefethen, Hausner2024}. ``Block" versions of iterative methods solve multiple systems with the same coefficient matrix. They often require fewer iterations to converge than solving each system sequentially would, because each system uses information from other systems to accelerate its convergence. These block methods are highly parallelizable with graphics processing units (\textbf{GPUs}), and the per-iteration solve time can grow sublinearly with respect to the number of systems solved at once. Thus, we seek to investigate the possibility of  solving copies of one system using block iterative methods, taking advantage of potential iteration reductions and slow-growing per-iteration solve times.

Here we apply this idea to conjugate gradient (\textbf{CG})~\cite{CG} on the system $Ax = b$ as a proof of concept. $A$ is an $n \times n$ symmetric positive definite (\textbf{SPD}) matrix, $b$ is a known vector, and $x$ is an unknown vector. We propose the Multi-Basis Parallel Conjugate Gradient (\textbf{MBPCG}), which solves multiple copies of a system using a variant of block conjugate gradient (\textbf{BCG}). Each copy has a different randomized initial guess. \cref{tab:notation} defines our notation and terminology.
\begin{table}
	\caption{Notation and terminology.}
	\label{tab:notation}
	\centering
	\begin{threeparttable}
		\renewcommand{\arraystretch}{1.3}
		\begin{tabularx}{\linewidth}{lX}
			\toprule
			Symbol                                                & Meaning                                                                                         \\
			\midrule
			nnz                                                   & number of nonzero entries in $A$. In dense matrices, $\nnz=n^2$                                 \\
			$k$                                                   & number of systems solved at once                                                                \\
			$A$ and $M$                                           & $n \times n$ matrix                                                                             \\
			any other capital Latin letter                        & $n \times k$ matrix                                                                             \\
			$h$, $i$, and $j$                                     & integer indices                                                                                 \\
			any other lower-case Latin letter in \cref{alg:mbpcg} & length-$n$ vector                                                                               \\
			Greek letter in \cref{alg:mbpcg}                      & $k \times k$ matrix                                                                             \\
			\~{} above symbol                                     & relating to the initial residual system                                                         \\[1ex]
			``speedup''                                           & $\dfrac{\text{CG time elapsed until convergence}}{\text{MBPCG time elapsed until convergence}}$ \\[1ex]
			``iteration count''                                   & iterations elapsed until convergence                                                            \\[1ex]
			``iteration ratio''                                   & $\dfrac{\text{MBPCG iteration count}}{\text{CG iteration count}}$                               \\[2ex]
			``per-iteration time ratio"                           & $\dfrac{\text{MBPCG time per iteration}}{\text{CG time per iteration}}$\tnote{a}                \\[1ex]
			superscript                                           & iteration count                                                                                 \\
			subscript next to a matrix                            & column vector                                                                                   \\
			\bottomrule
		\end{tabularx}
		\begin{tablenotes}
			\small
			\item[a] ``Time per iteration" is not simply $\frac{\text{total solve time}}{\text{iteration count}}$. It includes only \crefrange{line:Xi}{line:initial_tolerance} and \crefrange{line:W_zeta}{line:Sigma} of \cref{alg:mbpcg}. It is the average across all iterations except the last, as the last iteration executes \crefrange{line:hat_B}{line:initial_tolerance_endif} instead of \crefrange{line:W_zeta}{line:Sigma}.
		\end{tablenotes}
	\end{threeparttable}
\end{table}

Each iteration of CG consists of one sparse matrix-vector multiplication (\textbf{SpMV}), several vector operations such as scaling and addition, and several scalar operations. The inputs to each iteration depend on the outputs of the previous, and if computed on a GPU, only SpMV utilizes a significant portion of GPU resources. We thus consider a variant of CG proposed by Chronopoulos and Gear~\cite{ChronopoulosGear1989}, which has several independent and parallelizable steps. It is equivalent to standard CG in exact arithmetic, but sacrifices some numerical stability~\cite{Carson2018}. It accepts an initial guess vector $x^0$, which can be tuned to reduce iteration count.

Although the Chronopoulos-Gear variant is more parallelizable than standard CG, we hypothesize that there is further room to improve this using BCG. For example, CG's SpMV step reads $A$ from memory once per output vector. If we reuse $A$ to perform $k$ SpMV operations in parallel, each system only needs $1/k$ reads. The solve time per system thus grows sublinearly with $k$. Performing multiple parallel SpMV operations is equivalent to performing one sparse matrix-dense matrix multiplication (\textbf{SpMM}). More importantly, with CG, the iteration count can be high for large and relatively dense matrices. We hypothesize that our method can significantly reduce the iteration count, enough to offset the per-iteration time penalty.

The block conjugate gradient method is used to solve multiple systems that share a coefficient matrix $A$. Chronopoulos-Gear CG can be adopted into a BCG algorithm by replacing all vectors with $n \times k$ matrices and all scalars with $k \times k$ matrices. The equivalent system takes the form $AX=B$.

CG is often performed with a preconditioner to reduce iteration count. However, some linear systems require many iterations regardless of preconditioner choice~\cite{tunnell2025empiricalstudyconjugategradient}, leaving room for improvement. Since preconditioning is not our focus, we diagonally scale $A$ in lieu of preconditioning. Note that our method does not replace preconditioning and can be used with any preconditioner. We define the diagonal matrix $D$, where $\vphantom{\Big(}D_{jj}=\sqrt{A_{jj}}$. We then produce a diagonally-scaled system $\bar A \bar X = \bar B$, where $\bar A=D^{-1} A D^{-T}$, $\bar X = DX$, and $\bar B = D^{-1}B$. $X$ can then be recovered easily after solving this system. The residuals of the original system at the $i$th iteration, $R^i$, can be recovered from the residuals of the diagonally-scaled system, $\bar R^i$, by computing $R^i = D\bar R^i$. From here on, we will omit these steps and let $A$, $B$, $X^i$, and $R^i$ describe the diagonally-scaled system.

\section{The Multi-Basis Parallel Conjugate Gradient method}
Using a BCG-like approach to solve one linear system can reduce the iterations elapsed until convergence, as BCG explores multiple directions in each system's Krylov space at each iteration~\cite{JiLi2017BreakdownFree, JiLi2017Lstsq}. Each system utilizes residual information from all other systems to improve its approximation. Since BCG replaces vector operations with matrix operations, GPU resource usage per system is more efficient. 

For iteration $i$ and column index $j$, as $X^i_j$ converges toward the true value of $X$, $X^i$ becomes increasingly ill-conditioned. This is known as rank collapse and produces extremely small or extremely large entries in various intermediate matrices. This frequently causes floating-point overflow or high numerical errors in BCG, causing failure to converge~\cite{Tichy2025}. Thus, we adopt the ``Dubrulle-R BCG" described in Algorithm 5 of \cite{Tichy2025}, which maintains linear independence of the columns of each $n\times k$ matrix by performing QR decomposition in each iteration.

\cref{alg:mbpcg} describes our proposed MBPCG algorithm, which solves copies of one system using Dubrulle-R BCG. This idea can be adapted for other variants of BCG, as long as there is an orthogonalization step to prevent rank collapse. Our algorithm is compatible with standard preconditioning techniques through modifications resembling those in Algorithm 7 of~\cite{Tichy2025}.

\begin{algorithm}[htbp]
	\caption{Multi-Basis Parallel CG}
	\label{alg:mbpcg}
	\begin{algorithmic}[1]
		\setstretch{1.1}
		\State \textbf{Input} $A$, $b$ \label{line:input}
		\State $B = \begin{bmatrix}
		\vert & & \vert\\
		b & ... & b\\
		\vert&&\vert
		\end{bmatrix}$ \label{line:B}
		\State Randomly generate $U$ \label{line:U}
		\State $\vphantom{\Big(} X^0 = \frac{\| B\|_F}{\| A U\|_F}U$ \label{line:X^0} 
		\State $\tilde B = B - A X^0$ \label{line:tilde_B}
		\State $\tilde X^0 = \mathbf{O}$ \label{line:tilde_X^0}
		\State $R^0 = \tilde B$ \label{line:R^0}
		\State $W^0, \Sigma^0 = \text{qr}(R^0)$ \label{line:W^0_Sigma^0}
		\State $P^0 = W^0$ \label{line:P^0}
		\For{$i=1, 2, ...$} \label{line:loop}
		\State $\Xi^{i-1} = \left(P^{i-1, T} A P^{i-1}\right)^{-1}$ \label{line:Xi}
		\State $\tilde X^i = \tilde X^{i-1} + P^{i-1}\Xi^{i-1}\Sigma^{i-1}$ \label{line:tilde_X}
		\State $R^i = R^{i-1} - AP^{i-1}\Xi^{i-1}\Sigma^{i-1}$ \label{line:R}
		\State $\vphantom{\Big(} h = \arg\!\min_j \frac{\|R_j^i\|}{\|b\|}$ \label{line:h}
		\If {$\vphantom{\Big(} \frac{\|R_h^i\|}{\|b\|} < \text{initial tolerance}$} \label{line:initial_tolerance}
		\State $\hat B^i = B - R^i$ \label{line:hat_B}
		\State $\hat c^i=\left(\hat B^{i, T}\hat B^i\right)^{-1}(\hat B^{i, T})b$ \label{line:hat_c}
		\State $r^i = b - \hat B^i \hat c^i$ \label{line:r}
		\If {$\vphantom{\Big(}\frac{\|r^i\|}{\|b\|} < \text{convergence tolerance}$} \label{line:convergence_tolerance}
		\State exit \label{line:exit}
		\EndIf \label{line:convergence_tolerance_endif}
		\EndIf \label{line:initial_tolerance_endif}
		\State $W^i, \zeta^i = \text{qr}\left(W^{i-1}- AP^{i-1}\Xi^{i-1}\right)$ \label{line:W_zeta}
		\State $P^i = W^i + P^{i-1}\zeta^{i, T}$ \label{line:P}
		\State $\Sigma^i = \zeta^i\Sigma^{i-1}$ \label{line:Sigma}
		\EndFor \label{line:end_for}
		\State $X^i = \tilde X^i + X^0$ \label{line:X}
		\State $x^i = X^i \hat c^i$ \label{line:x}
		\State \textbf{Return} $x^i$ \label{line:return}
	\end{algorithmic}
\end{algorithm}

To improve numerical stability, we transform the original system into the ``initial residual system" $A \tilde X = \tilde B$, where $\tilde B := B - A X^0$ represents the residual from the randomized initial guess $X^0$. It follows that
\begin{equation}
	X = \tilde X + X^0. \label{eq:X_decomp}
\end{equation}
Thus, at the $i$th iteration, $X^i=\tilde X^i + X^0$. Note that the residual of the initial residual system is identical to $R^i$, meaning we do not need to explicitly recover $R^i$:
\begin{equation}
	R^i = B - A X^i = \left(\tilde B + A X^0\right) - \left(A \left(\tilde X^i + X^0\right)\right) = \tilde B - A \tilde X^i.                  
\end{equation}

The entries of $U$ are independent and identically distributed (i.i.d.) random variables generated with the distribution $\text{Uniform}(-1, 1)$. $U$ is then normalized to produce $X^0$ such that $\|AX^0\|_F = \|B\|_F$, ensuring the entries of $X^0$ land in a reasonable range.

When the relative error of the best basis, $\frac{\|R_h^i\|}{\|b\|}$, reaches an initial user-defined tolerance value, we compute the best linear combination of the columns of $X^i$ via least-squares on the system $AX^i c^i=b$. The least-squared solution is 
\begin{equation}
	\hat c^i := 
	\left(\left(AX^i\right)^T\left(AX^i\right)\right)^{-1}\left(AX^i\right)b.
\end{equation}

Now we show that $\hat B^i :=A X^i = B-R^i$, meaning we do not need to perform an expensive SpMM operation to find $\hat c^i$.

\begin{proposition}
	$AX^i = B-R^i$.
\end{proposition}
\begin{proof}
	By definition,
	\begin{align}
		\tilde{X}^i & = \tilde X^{i-1} + P^{i-1}\Xi^{i-1}\Sigma^{i-1} \label{eq:X_tilde} \\
		R^i         & = R^{i-1} - AP^{i-1}\Xi^{i-1}\Sigma^{i-1} \label{eq:R}.            
	\end{align}
	Multiplying both sides of \cref{eq:X_tilde} by $A$ and adding \cref{eq:X_tilde,eq:R} yields
	\begin{equation}
		A \tilde X^i + R^i = A \tilde X^{i-1} + R^{i-1}.\label{eq:prop_1_induction}
	\end{equation}
	Using the initial condition $\tilde X^0 = \mathbf{O}, R^0 = \tilde B$, we obtain
	\begin{align}
		A \tilde X^1 + R^{1}           & = \tilde B                                                                            \\
		A \tilde X^i+R^i               & = \tilde B \text{ by  \cref{eq:prop_1_induction} }                                    \\
		A\left(X^i - X^0 \right) + R^i & = B - A X^0 \text{ by \cref{eq:X_decomp} and \cref{line:tilde_B} of \cref{alg:mbpcg}} \\
		A X^i                          & = B - R^i                                                                             
	\end{align}
\end{proof}
When $\frac{\|r^i\|}{\|b\|}$ reaches a second user-defined tolerance value (``convergence tolerance"), the algorithm returns the final estimate of $x$. The two-tolerance approach ensures that the expensive least-squares step only runs when there is a sufficient possibility that the best linear combination's error falls below the convergence threshold. However, the least-squares step frequently resulted in numerical overflow or instability, as the operation $\vphantom{\Big(} \hat B^{i, T}\hat B^{i}$ can produce very large or very small values. Thus, our implementation falls back to using the relative error of the best linear combination if numerical issues arise and sets the two tolerance values to be equal. This way, the least squares step only runs once at the end.

\section{Implementation and performance analysis}
We implemented all iterative steps except SpMM with custom CUDA kernels, as they outperform vendor-provided libraries such as cuSolver and cuBLAS. All sparse matrices are stored in compressed sparse row format, and all dense matrices are stored column-major. Our implementation uses double-precision real floating-point numbers.

\cref{tab:runtime-by-step} shows that for sparse systems, the SpMM operation $AP^i$ is far more computationally expensive than all other iterative steps. We used an implementation of SpMM from~\cite{hypre}, which performs $k$ independent SpMV operations in parallel. There is potential to further optimize MBPCG by choosing faster SpMM algorithms that minimally increase execution time as $k$ increases. Row-major storage may also improve SpMM performance, as it would coalesce memory access across the rows of $P^i$, and uncoalesced memory access is a major cost in this SpMM implementation. For dense systems, we used cuSPARSE~\cite{cuSPARSE} for the dense matrix multiplication $AP^i$.

\begin{table}
	\centering
	\begin{threeparttable}
		\caption{Execution time of each MBPCG kernel for sparse system with matrix \texttt{Fault\_639}. Line numbers refer to \cref{alg:mbpcg}.\tnote{a} Some kernels correspond to multiple lines.}
		\label{tab:runtime-by-step}
		\begin{tabular}{lcrrrr}
			\toprule
			& & \multicolumn{4}{c}{Time (ms)} \\
			\cmidrule(lr){3-6}
			Variables computed                       & Lines                                      & $k=1$  & $k=2$  & $k=4$  & $k=8$  \\
			\midrule
			$A P^{i-1}$                              & \ref{line:Xi}                              & 0.5226 & 0.5821 & 0.8812 & 1.4880 \\
			$P^{i-1, T} \left(AP^{i-1}\right)$       & \ref{line:Xi}                              & 0.1363 & 0.1588 & 0.2047 & 0.7934 \\
			$P^{i-1} \Xi^{-1}$, $\tilde{X}^i$, $R^i$ & \ref{line:X}, \ref{line:R}                 & 0.0437 & 0.0763 & 0.1475 & 0.2853 \\
			$h$, $\|R^i_h\|$                         & \ref{line:h}, \ref{line:initial_tolerance} & 0.0361 & 0.0456 & 0.0572 & 0.0813 \\[0.15em]
			$W^{i-1}- AP^{i-1}\Xi^{i-1}$             & \ref{line:W_zeta}                          & 0.0365 & 0.0624 & 0.1302 & 0.4432 \\
			$W^i$, $\zeta^i$                         & \ref{line:W_zeta}                          & 0.0584 & 0.0760 & 0.1290 & 0.2579 \\
			$P^i$, $\Sigma^i$                        & \ref{line:P}, \ref{line:Sigma}             & 0.0306 & 0.0479 & 0.0862 & 0.1653 \\
			\bottomrule
		\end{tabular}
		\begin{tablenotes}
			\small
			\item[a] Timing each step separately requires GPU synchronizations and prevents overlapped kernel execution. Therefore, the sum of each step's solve times can exceed the fully optimized per-iteration solve times listed in \cref{sec:Appendix}.
		\end{tablenotes}
	\end{threeparttable}
\end{table}

We implemented multiplication by $\Xi^{i-1}$ using a Cholesky factorization of $\Xi^{i-1, -1}$, which is SPD. This is faster and more numerically stable than explicitly inverting $P^{i-1, T} A P^{i-1}$~\cite{Kushida2015,Trefethen}.

QR factorization of $n\times k$ matrices at \cref{line:W^0_Sigma^0,line:W_zeta} is computed using Cholesky QR as described in Figure 2 of~\cite{Fukaya2014}. Compared to methods such as Householder reflections and Givens rotations, Cholesky QR is more performant especially as $n \gg k$, but it is less numerically accurate~\cite{Fukaya2014}. Since our algorithm only requires linear independence of the columns, not strict orthogonality, the reduced accuracy is acceptable.

\section{Experimental method}
To determine the effectiveness of MBPCG, we compared iteration count and solve time between MBPCG and Chronopoulos-Gear CG in both sparse and dense linear systems. CG will hereafter refer to the Chronopoulos-Gear conjugate gradient. We conducted numerical experiments on one NVIDIA Tesla V100 GPU with 16GB of memory.

Our CG implementation uses the zero vector as its initial guess, while MBPCG's initial guesses do not include the zero vector, even when $k=1$. We chose not to standardize this because MBPCG does not converge for some small systems if $X^0$ has a zero column. Furthermore,  \cref{fig:sparse-trajectories,fig:dense-trajectories} show that the error trajectories for CG and $k=1$ are nearly identical despite this difference in initial guesses. This means that using randomized guesses in CG adds unnecessary steps without producing an iteration reduction.

We conducted experiments on 107 sparse SPD matrices from the SuiteSparse matrix collection. These matrices arise from real-world linear algebra problems in fields such as fluid dynamics, structural engineering, robotics, and grid optimization~\cite{SuiteSparse}. All right-hand side vectors ($b$) were generated randomly, with i.i.d. entries under the distribution $\text{Uniform}(-1, 1)$. We solved each system using CG, then MBPCG with $k$ set to 1, 2, 4, and 8. We then recorded iteration counts and solve times. We used a tolerance value of $10^{-8}$, one of several commonly used values~\cite{ZhengDong2011, Belyaeva2021, Bergamaschi2006, Ma2024}.

Since there is no public collection of real-world large and dense SPD matrices, we randomly generated $A$ in our dense experiments. Specifically, we randomly generated the $n \times n$ matrix $M$ with i.i.d. $\text{Uniform}(-1, 1)$ entries and set $A=M^TM$. $b$ has i.i.d. $\text{Uniform}(-1, 1)$ random entries.

In our implementation, solve times are significantly higher when $k\geq16$. Due to constraints such as the number of registers, $k \geq 16$ requires different kernel designs with considerable performance penalties. Our implementation is also optimized for $k$ being a power of 2.

We included the diagonal scaling steps in our timing.

\section{Experimental results}
In sparse systems tested, we observed moderate to significant reductions in iteration count as $k$ increases, which \cref{fig:sparse-trajectories} demonstrates. Note that CG has an effectively identical error trajectory as MBPCG at $k=1$.

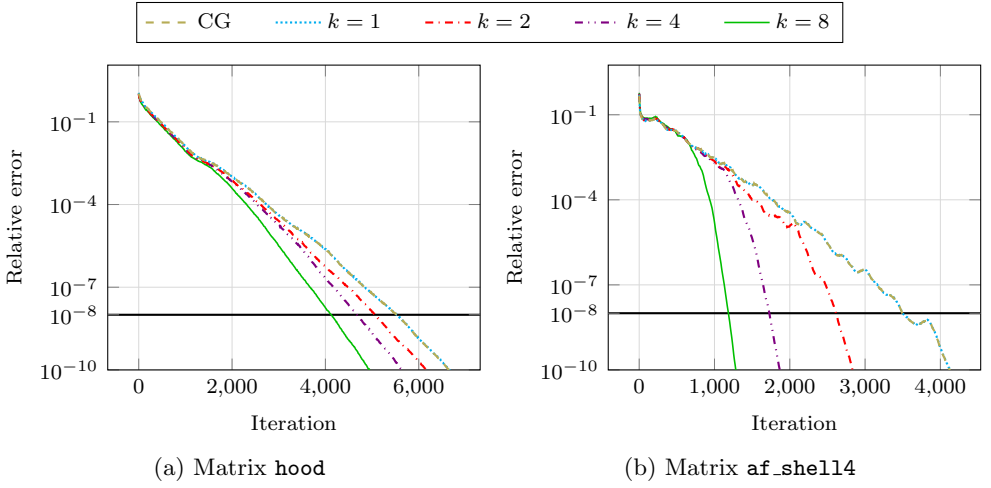
\begin{figure}[htbp]
	\centering
	\begin{tikzpicture}
		\begin{axis}[
				hide axis,
				xmin=10, xmax=50, ymin=0, ymax=0.4,
				legend columns=-1,
				legend style={
					font=\footnotesize,
					/tikz/every even column/.style={column sep=0.5cm},
					at={(0.5,0.5)},
					anchor=center
				},
				legend reversed,
			]
			\addlegendimage{green!75!black, solid, mark=none, semithick}
			\addlegendentry{$k=8$}
			\addlegendimage{violet, dash dot dot, mark=none, thick}
			\addlegendentry{$k=4$}
			\addlegendimage{red, dash dot, mark=none, thick}
			\addlegendentry{$k=2$}
			\addlegendimage{cyan, densely dotted, mark=none, thick}
			\addlegendentry{$k=1$}
			\addlegendimage{yellow!65!black, dashed, mark=none, thick}
			\addlegendentry{CG}
		\end{axis}
	\end{tikzpicture}
	\par
	\vspace{0.25cm}
	\subfloat[Matrix \texttt{hood}]{%
		\begin{tikzpicture}
			\begin{axis}[
					width=0.50\linewidth,
					xlabel={Iteration},
					ylabel={Relative error},
					label style={font=\footnotesize},
					tick label style={font=\footnotesize},
					grid=major,
					grid style={very thin, gray!30},
					scaled x ticks=false,
					xticklabel style={/pgf/number format/fixed},
					ymin=1e-10,
					ymode=log,
					extra y ticks={1E-8},
					extra y tick labels={$10^{-8}$},
					extra y tick style={grid=major, grid style={thick, black, solid}},
					cycle list={
						{green!75!black, solid, mark=none, semithick},
						{violet, dash dot dot, mark=none, thick},
						{red, dash dot, mark=none, thick},
						{cyan, densely dotted, mark=none, thick},
						{yellow!65!black, dashed, mark=none, thick}
					}
				]
				\pgfplotsinvokeforeach{4,3,...,0} {
					\addplot+[
						unbounded coords=discard,
						filter discard warning=false,
						each nth point={10}
						] table [
						col sep=comma,
						x expr=\coordindex+1,
						y index=#1
					] {data/hood.csv};
				}
			\end{axis}
		\end{tikzpicture}%
	}%
	\hfill%
	\subfloat[Matrix \texttt{af\_shell4}]{%
		\begin{tikzpicture}
			\begin{axis}[
					width=0.50\linewidth,
					xlabel={Iteration},
					ylabel={Relative error},
					label style={font=\footnotesize},
					tick label style={font=\footnotesize},
					xticklabel style={/pgf/number format/fixed},
					grid=major,
					grid style={very thin, gray!30},
					ymin=1e-10,
					ymode=log,
					extra y ticks={1E-8},
					extra y tick labels={$10^{-8}$},
					extra y tick style={grid=major, grid style={thick, black, solid}},
					cycle list={
						{green!75!black, solid, mark=none, semithick},
						{violet, dash dot dot, mark=none, thick},
						{red, dash dot, mark=none, thick},
						{cyan, densely dotted, mark=none, thick},
						{yellow!65!black, dashed, mark=none, thick}
					}
				]
				\pgfplotsinvokeforeach{4,3,...,0} {
					\addplot+[
						unbounded coords=discard,
						filter discard warning=false,
						each nth point={6}
						] table [
						col sep=comma,
						x expr=\coordindex+1,
						y index=#1
					] {data/af_shell4.csv};
				}
			\end{axis}
		\end{tikzpicture}%
	}
	\caption{Error trajectories across iterations in select sparse systems. Increasing $k$ usually reduces iteration count, but the amount varies by matrix.}
	\label{fig:sparse-trajectories}
\end{figure}

The balance between iteration reduction and increased per-iteration cost means that the best $k$ depends on properties of the system. In most systems tested, either CG or $k=8$ is the best strategy. $k=4$ is the best in 11 of the systems tested, and $k=2$ is the best in one of them. The appendix shows the best solve strategy for each system.

We tested 15 dense systems, with $n$ between 40 and 19{,}200. Dense systems see significant iteration reductions, often more so than sparse systems. \cref{fig:dense-trajectories} shows examples of error trajectories. The speedup is also generally higher, and $k=8$ is the best strategy in all systems tested.

\begin{figure}[htbp]
	\centering
	\begin{tikzpicture}
		\begin{axis}[
				hide axis,
				xmin=10, xmax=50, ymin=0, ymax=0.4,
				legend columns=-1,
				legend style={
					font=\footnotesize,
					/tikz/every even column/.style={column sep=0.5cm},
					at={(0.5,0.5)},
					anchor=center
				},
				legend reversed,
			]
			\addlegendimage{green!75!black, solid, mark=none, semithick}
			\addlegendentry{$k=8$}
			\addlegendimage{violet, dash dot dot, mark=none, thick}
			\addlegendentry{$k=4$}
			\addlegendimage{red, dash dot, mark=none, thick}
			\addlegendentry{$k=2$}
			\addlegendimage{cyan, densely dotted, mark=none, thick}
			\addlegendentry{$k=1$}
			\addlegendimage{yellow!65!black, dashed, mark=none, thick}
			\addlegendentry{CG}
		\end{axis}
	\end{tikzpicture}
	\par
	\vspace{0.25cm}
	\subfloat[$n=4{,}800$]{%
		\begin{tikzpicture}
			\begin{axis}[
					width=0.48\linewidth,
					xlabel={Iteration},
					ylabel={Relative error},
					label style={font=\footnotesize},
					tick label style={font=\footnotesize},
					grid=major,
					grid style={very thin, gray!30},
					scaled x ticks=false,
					xticklabel style={/pgf/number format/fixed},
					xtick distance=2500,
					ymin=1e-10,
					ymode=log,
					extra y ticks={1E-8},
					extra y tick labels={$10^{-8}$},
					extra y tick style={grid=major, grid style={thick, black, solid}},
					cycle list={
						{green!75!black, solid, mark=none, semithick},
						{violet, dash dot dot, mark=none, thick},
						{red, dash dot, mark=none, thick},
						{cyan, densely dotted, mark=none, thick},
						{yellow!65!black, dashed, mark=none, thick}
					}
				]
				\pgfplotsinvokeforeach{4,3,...,0} {
					\addplot+[
						unbounded coords=discard,
						filter discard warning=false,
						each nth point={23}
						] table [
						col sep=comma,
						x expr=\coordindex+1,
						y index=#1
					] {data/dense_4800.csv};
				}
			\end{axis}
		\end{tikzpicture}%
	}%
	\hfill%
	\subfloat[$n=16{,}000$]{%
		\begin{tikzpicture}
			\begin{axis}[
					width=0.48\linewidth,
					xlabel={Iteration},
					ylabel={Relative error},
					label style={font=\footnotesize},
					tick label style={font=\footnotesize},
					xticklabel style={/pgf/number format/fixed},
					grid=major,
					grid style={very thin, gray!30},
					scaled x ticks=false,
					xticklabel style={/pgf/number format/fixed},
					ymin=1e-10,
					ymode=log,
					extra y ticks={1E-8},
					extra y tick labels={$10^{-8}$},
					extra y tick style={grid=major, grid style={thick, black, solid}},
					cycle list={
						{green!75!black, solid, mark=none, semithick},
						{violet, dash dot dot, mark=none, thick},
						{red, dash dot, mark=none, thick},
						{cyan, densely dotted, mark=none, thick},
						{yellow!65!black, dashed, mark=none, thick}
					}
				]
				\pgfplotsinvokeforeach{4,3,...,0} {
					\addplot+[
						unbounded coords=discard,
						filter discard warning=false,
						each nth point={4}
						] table [
						col sep=comma,
						x expr=\coordindex * 20 + 1, 
						y index=#1
					] {data/dense_16000_small.csv};
				}
			\end{axis}
		\end{tikzpicture}%
	}
	\caption{Error trajectories across iterations in select dense systems. Increasing $k$ reduces iteration count significantly, but with diminishing returns.}
	\label{fig:dense-trajectories}
\end{figure}
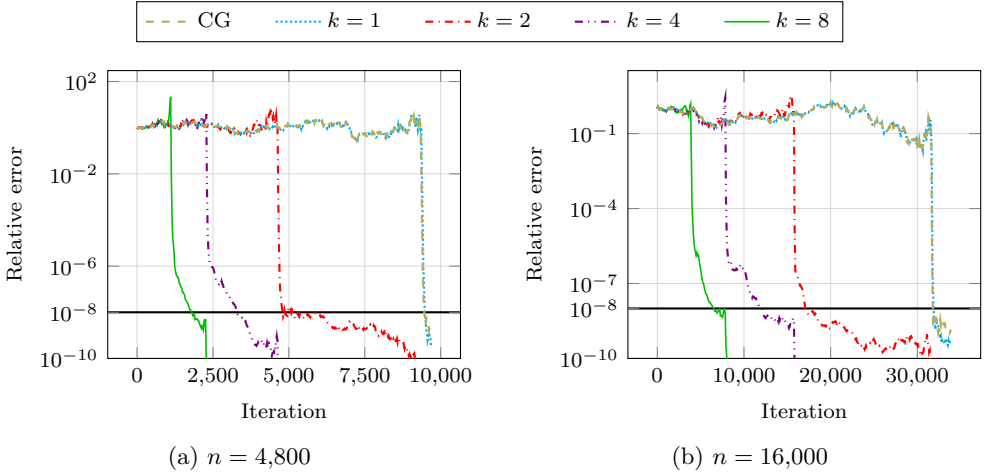

Although we implemented the sparse and dense solvers differently, it is worth investigating the two sets of results as one combined dataset to identify patterns across systems with drastically different properties. Let $t$ denote the solve time using CG. \cref{tab:regression-it,tab:regression-speedup-3-input} present the best multiple linear regression models we found for speedup and iteration ratio with a small number of inputs. We model $\log_{10}(\text{speedup})$ with $\log_{10}(n)$, $\log_{10}(\nnz)$, and $\log_{10}(t)$. Regarding iteration ratio, $\log_{10}(\nnz)$ does not have a statistically significant contribution and is dropped, resulting in a two-input regression model. We chose to measure iteration reduction using iteration ratio instead of its inverse because we did not find a good regression model for its inverse. Since $k=2$ is almost never the best strategy, we omitted it from these tables. The parameters of the regression models will likely differ if experiments are performed on different hardware and implementation, but we aim to show that a pattern for speedup potential exists.

\begin{table}
	\centering
	\begin{threeparttable}
		\caption{Regression models for iteration ratio. It is associated positively with $\log_{10}(n)$ and negatively with $\log_{10}(t)$.}
		\label{tab:regression-it}
		\renewcommand{\arraystretch}{1.3}
		\begin{tabular}{c r @{\enskip{}=\enskip} S[table-format=-1.4] S[table-format=-1.4] S[table-format=-1.4] c}
			\toprule
			\multicolumn{6}{l}{$\bm{k = 4}$ \hfill $R^2 = 0.6005$, adjusted $R^2 = 0.5938$} \\
			\midrule
			\multirow{2}{*}{Input} & \multicolumn{2}{c}{\multirow{2}{*}{Coefficient}} & \multicolumn{2}{c}{95\% confidence interval} & \multirow{2}{*}{P-value} \\
			\cmidrule(lr){4-5}
			& \multicolumn{2}{c}{} & {Lower bound} & {Upper bound} & \\
			\midrule
			constant       & $\beta_0$ & 0.1713  & 0.0595  & 0.2831  & \num{2.96e-\03} \\
			$\log_{10}(n)$ & $\beta_n$ & 0.1812  & 0.1505  & 0.2118  & \num{1.66e-21}  \\
			$\log_{10}(t)$ & $\beta_t$ & -0.1469 & -0.1715 & -0.1223 & \num{8.12e-22}  \\
			\midrule
			            
			\midrule
			\multicolumn{6}{l}{$\bm{k = 8}$ \hfill $R^2 = 0.6513$, adjusted $R^2 = 0.6454$} \\
			\midrule
			\multirow{2}{*}{Input} & \multicolumn{2}{c}{\multirow{2}{*}{Coefficient}} & \multicolumn{2}{c}{95\% confidence interval} & \multirow{2}{*}{P-value} \\
			\cmidrule(lr){4-5}
			& \multicolumn{2}{c}{} & {Lower bound} & {Upper bound} & \\
			\midrule
			constant       & $\beta_0$ & 0.0294  & -0.0893 & 0.1481  & \num{6.25e-\01} \\
			$\log_{10}(n)$ & $\beta_n$ & 0.2035  & 0.1709  & 0.2360  & \num{4.23e-23}  \\
			$\log_{10}(t)$ & $\beta_t$ & -0.1812 & -0.2073 & -0.1551 & \num{2.64e-26}  \\
			\bottomrule
		\end{tabular}
	\end{threeparttable}
\end{table}

\begin{table}
	\centering
	\begin{threeparttable}
		\caption{Regression models for $\log_{10}(\text{per-iteration time ratio})$. It is associated positively with $\log_{10}(n)$ and negatively with $\log_{10}(\nnz)$.}
		\label{tab:regression-time-per-it}
		\renewcommand{\arraystretch}{1.3}
		\begin{tabular}{c r @{\enskip{}=\enskip} S[table-format=-1.4] S[table-format=-1.4] S[table-format=-1.4] c}
			\toprule
			\multicolumn{6}{l}{$\bm{k = 4}$ \hfill $R^2 = 0.6588$, adjusted $R^2 = 0.6531$} \\
			\midrule
			\multirow{2}{*}{Input} & \multicolumn{2}{c}{\multirow{2}{*}{Coefficient}} & \multicolumn{2}{c}{95\% confidence interval} & \multirow{2}{*}{P-value} \\
			\cmidrule(lr){4-5}
			& \multicolumn{2}{c}{} & {Lower bound} & {Upper bound} & \\
			\midrule
			constant          & $\beta_0$ & 0.1315  & 0.0973  & 0.1657  & \num{6.71e-12}  \\
			$\log_{10}(n)$    & $\beta_n$ & 0.0899  & 0.0781  & 0.1018  & \num{2.99e-29}  \\
			$\log_{10}(\nnz)$ & $\beta_z$ & -0.0538 & -0.0635 & -0.0441 & \num{7.67e-20}  \\
			\midrule
			            
			\midrule
			\multicolumn{6}{l}{$\bm{k = 8}$ \hfill $R^2 = 0.8010$, adjusted $R^2 = 0.7977$} \\
			\midrule
			\multirow{2}{*}{Input} & \multicolumn{2}{c}{\multirow{2}{*}{Coefficient}} & \multicolumn{2}{c}{95\% confidence interval} & \multirow{2}{*}{P-value} \\
			\cmidrule(lr){4-5}
			& \multicolumn{2}{c}{} & {Lower bound} & {Upper bound} & \\
			\midrule
			constant          & $\beta_0$ & -0.0244 & -0.0855 & 0.0367  & \num{4.31e-\01} \\
			$\log_{10}(n)$    & $\beta_n$ & 0.2077  & 0.1864  & 0.2289  & \num{1.26e-38}  \\
			$\log_{10}(\nnz)$ & $\beta_z$ & -0.0858 & -0.1031 & -0.0685 & \num{5.52e-17}  \\
			\bottomrule
		\end{tabular}
	\end{threeparttable}
\end{table}

\begin{table}
	\centering
	\begin{threeparttable}
		\caption{Three-input regression models for $\log_{10}(\text{speedup})$. It is associated negatively with $\log_{10}(n)$ and positively with $\log_{10}(\nnz)$ and $\log_{10}(t)$.}
		\label{tab:regression-speedup-3-input}
		\renewcommand{\arraystretch}{1.3}
		\begin{tabular}{c r @{\enskip{}=\enskip} S[table-format=-1.4] S[table-format=-1.4] S[table-format=-1.4] c}
			\toprule
			\multicolumn{6}{l}{$\bm{k = 4}$ \hfill $R^2 = 0.7240$, adjusted $R^2 = 0.7169$} \\
			\midrule
			\multirow{2}{*}{Input} & \multicolumn{2}{c}{\multirow{2}{*}{Coefficient}} & \multicolumn{2}{c}{95\% confidence interval} & \multirow{2}{*}{P-value} \\
			\cmidrule(lr){4-5}
			& \multicolumn{2}{c}{} & {Lower bound} & {Upper bound} & \\
			\midrule
			constant          & $\beta_0$ & 0.4882  & 0.3659  & 0.6105  & \num{1.59e-12}  \\
			$\log_{10}(n)$    & $\beta_n$ & -0.2538 & -0.2926 & -0.2150 & \num{2.04e-24}  \\
			$\log_{10}(\nnz)$ & $\beta_z$ & 0.0512  & 0.0103  & 0.0922  & \num{1.47e-\02} \\
			$\log_{10}(t)$    & $\beta_t$ & 0.1290  & 0.1010  & 0.1571  & \num{2.53e-15}  \\
			\midrule
			    
			\midrule
			\multicolumn{6}{l}{$\bm{k = 8}$ \hfill $R^2 = 0.7935$, adjusted $R^2 = 0.7883$} \\
			\midrule
			\multirow{2}{*}{Input} & \multicolumn{2}{c}{\multirow{2}{*}{Coefficient}} & \multicolumn{2}{c}{95\% confidence interval} & \multirow{2}{*}{P-value} \\
			\cmidrule(lr){4-5}
			& \multicolumn{2}{c}{} & {Lower bound} & {Upper bound} & \\
			\midrule
			constant          & $\beta_0$ & 0.9402  & 0.7738  & 1.1066  & \num{2.98e-20}  \\
			$\log_{10}(n)$    & $\beta_n$ & -0.4249 & -0.4777 & -0.3721 & \num{3.28e-31}  \\
			$\log_{10}(\nnz)$ & $\beta_z$ & 0.0724  & 0.0167  & 0.1282  & \num{1.13e-\02} \\
			$\log_{10}(t)$    & $\beta_t$ & 0.1942  & 0.1561  & 0.2324  & \num{1.27e-17}  \\
			\bottomrule
		\end{tabular}
	\end{threeparttable}
\end{table}

Speedup depends on nnz, but iteration ratio does not. We can therefore deduce that nnz is negatively correlated with per-iteration time ratio for $k>1$ (i.e., the per-iteration penalty from solving multiple systems). \cref{tab:regression-time-per-it} shows that our experiments support this conclusion. Per-iteration time ratio does not have a statistically significant linear dependence on $t$. This is intuitive, as the computations performed at each operation only depend on $A$'s size and sparsity structure. $n$ and nnz already account for these factors, making $t$ redundant.

CG solve time depends heavily on the condition number and eigenvalue distribution of $A$, in addition to $n$ and nnz~\cite{Trefethen, Hausner2024}. The dependence of speedup and iteration ratio on CG solve time thus means that difficult-to-solve systems benefit the most from MBPCG. CG solve time can only be known after solving the system. However, in applications where systems with identical sparsity structures are repeatedly solved, the first solve may provide useful information to predict whether MBPCG can speed up subsequent solves. Although speedup has a negative linear relationship with $n$, MBPCG is not necessarily worse for systems with large $n$. These systems, on average, have more nonzeros and take longer to solve.

To aid visualization of the three-input models for $\log_{10}(\text{speedup})$, we can create equivalent two-input models. Letting $\hat y$ denote the predicted speedup, we have:
\begin{align}
	\begin{split}
	\hat y & = \beta_0 + \beta_n \log_{10}(n) + \beta_z\log_{10}(\nnz) + \beta_t\log_{10}(t)      \\
	       & = \beta_0 + \beta_z \log_{10}\left(\frac{\nnz}{n^\rho}\right) + \beta_t\log_{10}(t), 
	\end{split}
\end{align}
where $\rho = -{\beta_n}/{\beta_z}$.

$\nnz/n$ and $\nnz/n^2$ are common measures of square matrix density~\cite{Booth2024, Stylianou2023, Giannoula2022, Malik2020}. By extension, we define $d=\nnz/n^\rho$ as ``quasi-density" if $\rho>0$. We can then describe $\hat y$ with a two-input linear model with inputs $\log_{10}(d)$ and $\log_{10}(t)$. $\rho$ varies with $k$, as it characterizes the quasi-density measure that empirically yields the best regression model for a particular $k$. \cref{tab:regression-speedup-2-input} presents the two-input models. Note that the coefficients of the constant term and $\log_{10}(t)$ are identical to those in the three-input tables, and the coefficient of $\log_{10}(\nnz)$ is identical to that of $\log_{10}(d)$.

\begin{table}
	\centering
	\begin{threeparttable}
		\caption{Two-input regression models for $\log_{10}(\text{speedup})$. It is associated positively with $\log_{10}(d)$ and $\log_{10}(t)$.}
		\label{tab:regression-speedup-2-input}
		\centering
		\renewcommand{\arraystretch}{1.3}
		\begin{tabular}{c r @{\enskip{}=\enskip} S[table-format=-1.4] S[table-format=-1.4] S[table-format=-1.4] c}
			\toprule
			\multicolumn{6}{l}{$\bm{k = 4}$ \hfill $R^2 = 0.7240$, adjusted $R^2 = 0.7193$, $\rho = 4.9552$} \\
			\midrule
			\multirow{2}{*}{Input} & \multicolumn{2}{c}{\multirow{2}{*}{Coefficient}} & \multicolumn{2}{c}{95\% confidence interval} & \multirow{2}{*}{P-value} \\
			\cmidrule(lr){4-5}
			& \multicolumn{2}{c}{} & {Lower bound} & {Upper bound} & \\
			\midrule
			constant       & $\beta_0$ & 0.4882 & 0.4017 & 0.5747 & \num{2.91e-20} \\
			$\log_{10}(d)$ & $\beta_d$ & 0.0512 & 0.0451 & 0.0574 & \num{1.31e-32} \\
			$\log_{10}(t)$ & $\beta_t$ & 0.1290 & 0.1092 & 0.1488 & \num{2.40e-24} \\
			\midrule
			            
			\midrule
			\multicolumn{6}{l}{$\bm{k = 8}$ \hfill $R^2 = 0.7935$, adjusted $R^2 = 0.7901$, $\rho = 5.8654$} \\
			\midrule
			\multirow{2}{*}{Input} & \multicolumn{2}{c}{\multirow{2}{*}{Coefficient}} & \multicolumn{2}{c}{95\% confidence interval} & \multirow{2}{*}{P-value} \\
			\cmidrule(lr){4-5}
			& \multicolumn{2}{c}{} & {Lower bound} & {Upper bound} & \\
			\midrule
			constant       & $\beta_0$ & 0.9402 & 0.8201 & 1.0603 & \num{2.52e-30} \\
			$\log_{10}(d)$ & $\beta_d$ & 0.0724 & 0.0655 & 0.0793 & \num{2.12e-41} \\
			$\log_{10}(t)$ & $\beta_t$ & 0.1942 & 0.1670 & 0.2215 & \num{3.65e-27} \\
			\bottomrule
		\end{tabular}
	\end{threeparttable}
\end{table}

\cref{fig:scatter} displays the above performance metrics with respect to their best regressors. The choice between $k=4$ and $k=8$ usually does not affect whether $\text{speedup}>1$. However, $k=8$ has a greater best-case speedup potential and a greater worse-case slowdown potential than $k=4$. Given a linear system, one can determine the strategy to solve it using regression models in \cref{tab:regression-speedup-3-input}. We do not consider $k=2$, because it is best in only one of the systems tested. For dense systems, $k=8$ is best in all of our experiments.
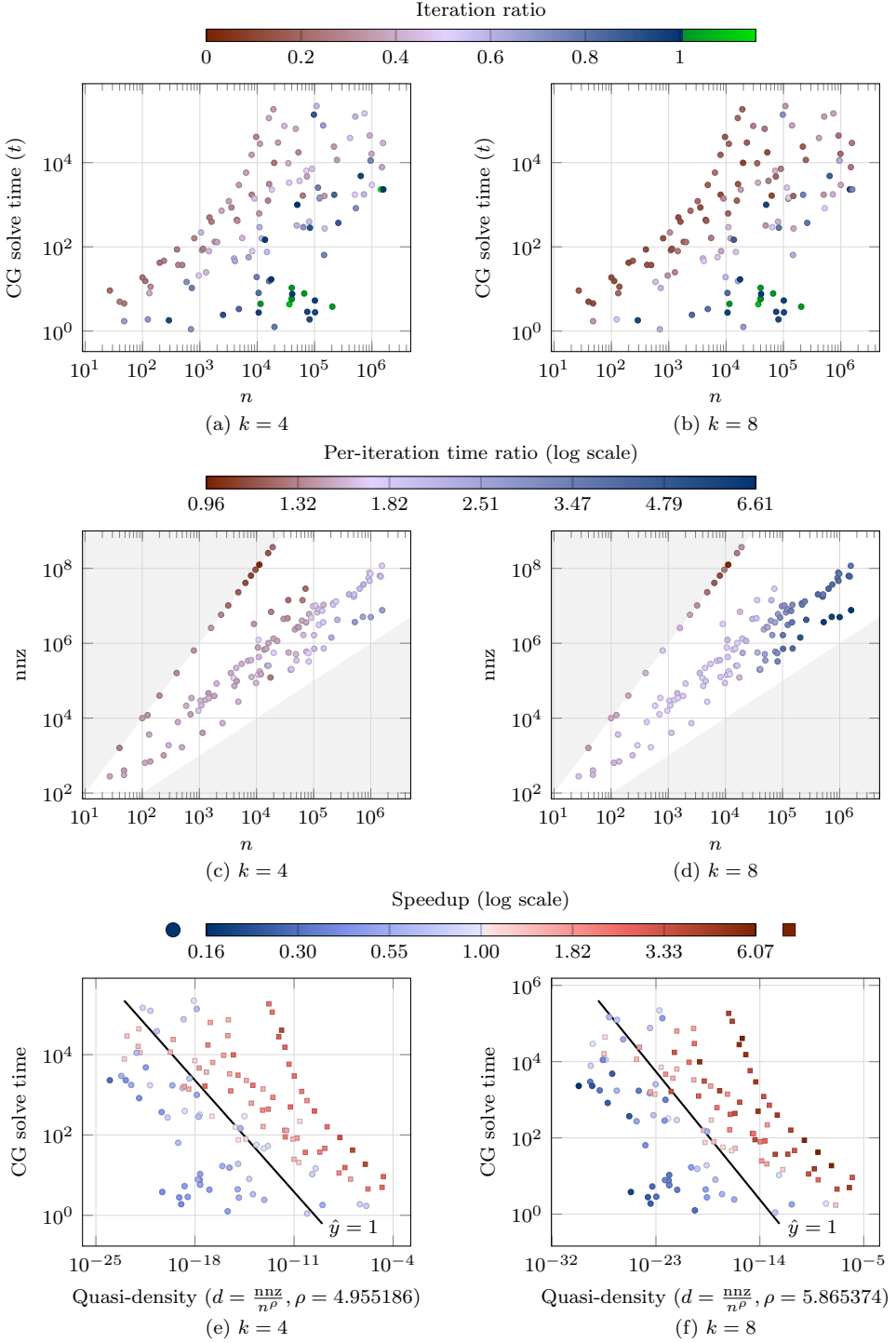
\begin{figure}[htbp]
	\centering
	\begin{tikzpicture}
		\begin{groupplot}[
				group style={
					group size=2 by 3,
					horizontal sep=2.0cm,
					vertical sep=2.55cm,
				},
				width=0.48\linewidth,
				label style={font=\footnotesize},
				tick label style={font=\footnotesize},
				grid=major,
				grid style={very thin, gray!30},
				every axis plot/.append style={mark size=1.1pt}
			]
			
			\nextgroupplot[
				xlabel={$n$},
				ylabel={CG solve time ($t$)},
				ylabel near ticks,
				xmode=log,
				ymode=log,
				x label style={font=\footnotesize},
				y label style={font=\footnotesize},
				log basis x=10,
				point meta min=0,
				point meta max=1.16,
				colormap name=it_ratio_colors,
			]
			\addplot[scatter, only marks, scatter src=explicit, unbounded coords=discard] 
			table [col sep=comma, empty cells with={nan}, x={n}, y={Regular CG Time}, meta={k=4 It Ratio}] {data/experiment_results.csv};
			
			\nextgroupplot[
				xlabel={$n$},
				ylabel={CG solve time ($t$)},
				ylabel near ticks,
				xmode=log,
				ymode=log,
				x label style={font=\footnotesize},
				y label style={font=\footnotesize},
				log basis x=10,
				point meta min=0,
				point meta max=1.16,
				colormap name=it_ratio_colors,
			]
			\addplot[scatter, only marks, scatter src=explicit, unbounded coords=discard] 
			table [col sep=comma, empty cells with={nan}, x={n}, y={Regular CG Time}, meta={k=8 It Ratio}] {data/experiment_results.csv};
			
			\nextgroupplot[
				xlabel={$n$},
				ylabel={\nnz},
				ylabel near ticks,
				x label style={font=\footnotesize},
				y label style={font=\footnotesize},
				xmode=log,
				ymode=log,
				xmin=9e0,
				ymin=7e1,
				ymax=1e9,
				point meta min=-0.02, point meta max=0.82,
				colormap name=time_per_it_ratio_colors,
				enlargelimits=false,
				clip=true,
				axis on top=true
			]
			\addplot[scatter, only marks, scatter src=explicit] 
			table [col sep=comma, x={n}, y={nnz}, meta={k=4 Log It Time Ratio}] {data/experiment_results.csv};
			\addplot [ 
				domain=1:5e6,
				samples=2,
				fill=gray!10,
				draw=none,
				forget plot
			] {x} \closedcycle;   
			\path[fill=gray!10] 
			plot[domain=0:6.6989, samples=2] (axis cs:{10^\x}, {(10^\x)^2}) 
			-- (rel axis cs:1, 1)
			-- (rel axis cs:0, 1)
			-- cycle;
			
			\nextgroupplot[
				xlabel={$n$},
				ylabel={\nnz},
				ylabel near ticks,
				x label style={font=\footnotesize},
				y label style={font=\footnotesize},
				xmode=log,
				ymode=log,
				xmin=9e0,
				ymin=7e1,
				ymax=1e9,
				point meta min=-0.02, point meta max=0.82,
				colormap name=time_per_it_ratio_colors,
				enlargelimits=false,
				clip=true,
				axis on top=true
			]
			\addplot[scatter, only marks, scatter src=explicit] 
			table [col sep=comma, x={n}, y={nnz}, meta={k=8 Log It Time Ratio}] {data/experiment_results.csv};
			\addplot [ 
				domain=1:5e6,
				samples=2,
				fill=gray!10,
				draw=none,
				forget plot
			] {x} \closedcycle;   
			\path[fill=gray!10] 
			plot[domain=0:6.6989, samples=2] (axis cs:{10^\x}, {(10^\x)^2}) 
			-- (rel axis cs:1, 1)
			-- (rel axis cs:0, 1)
			-- cycle;
			
			\nextgroupplot[
				xlabel style={align=center},
				xlabel={
					Quasi-density ($d=\frac{\nnz}{n^{\rho}}, \rho=4.955186$)
				},
				ylabel={CG solve time}, ylabel near ticks,
				x label style={font=\footnotesize},
				y label style={font=\footnotesize},
				xmode=log, ymode=log,
				point meta min=-0.783, point meta max=0.783,
				colormap name=speedup_colors,
			]
			\addplot[
				scatter, only marks, scatter src=explicit,
				visualization depends on={value \thisrow{k=4 Log Speedup} \as \logspeedup},
				scatter/@pre marker code/.append code={
					\pgfmathparse{\logspeedup > 0}
					\ifnum\pgfmathresult=1
						\def\markopts{mark=square*, mark size=0.8635pt}
					\else
						\def\markopts{mark=*}
					\fi
					\expandafter\scope\expandafter[\markopts]
				},
				scatter/@post marker code/.append code={
					\endscope
				}
			] 
			table [col sep=comma, x={k=4 Speedup Density}, y={Regular CG Time}, meta={k=4 Log Speedup}] {data/experiment_results.csv};
			
			\addplot[domain=1e-23 : 1e-9, samples=5, thick, black, forget plot] 
			{10^(-(0.488191/0.129015)) * x^(-(0.051223/0.129015))}
			node [anchor=south east, font=\footnotesize] at (rel axis cs:0.93, 0) {$\hat y = 1$};
			
			\nextgroupplot[
				xlabel style={align=center},
				xlabel={
					Quasi-density ($d=\frac{\nnz}{n^{\rho}}, \rho=5.865374$)
				},
				ylabel={CG solve time}, ylabel near ticks,
				x label style={font=\footnotesize},
				y label style={font=\footnotesize},
				xmode=log, ymode=log,
				point meta min=-0.783, point meta max=0.783,
				colormap name=speedup_colors,
			]
			\addplot[
				scatter, only marks, scatter src=explicit,
				visualization depends on={value \thisrow{k=8 Log Speedup} \as \logspeedup},
				scatter/@pre marker code/.append code={
					\pgfmathparse{\logspeedup > 0}
					\ifnum\pgfmathresult=1
						\def\markopts{mark=square*, mark size=0.8635pt} 
					\else
						\def\markopts{mark=*}
					\fi
					\expandafter\scope\expandafter[\markopts]
				},
				scatter/@post marker code/.append code={
					\endscope
				}
			] 
			table [col sep=comma, x={k=8 Speedup Density}, y={Regular CG Time}, meta={k=8 Log Speedup}] {data/experiment_results.csv};
			
			\addplot[domain=1e-28 : 5e-13, samples=5, thick, black, forget plot] 
			{10^(-(0.940219/0.194237)) * x^(-(0.072440/0.194237))}
			node [anchor=south east, font=\footnotesize] at (rel axis cs:0.9, 0) {$\hat y = 1$};
			
		\end{groupplot}
		        
		\path (group c1r1.north) -- (group c2r1.north) coordinate[midway] (mid1);
		\begin{axis}[
				at={(mid1)}, yshift=0.6cm, anchor=south,
				hide axis, scale only axis, height=0pt, width=0.5\linewidth,
				point meta min=0, point meta max=1.16,
				colormap name=it_ratio_colors,
				colorbar horizontal,
				colorbar/width=1.75mm,
				colorbar style={
					at={(0.5,0)}, anchor=south,
					width=0.6\linewidth,
					title={Iteration ratio},
					title style={font=\footnotesize, yshift=-1.5mm},
					tick style={black, opacity=0.2, line width=0.75pt},
					tick label style={font=\footnotesize},
					yticklabel={%
						\pgfmathparse{\tick}%
						\pgfmathprintnumber[fixed, fixed zerofill, precision=2]{\pgfmathresult}%
					}
				}
			]
		\end{axis}
		
		\path (group c1r2.north) -- (group c2r2.north) coordinate[midway] (mid2);
		\begin{axis}[
				at={(mid2)}, yshift=0.6cm, anchor=south,
				hide axis, scale only axis, height=0pt, width=0.5\linewidth,
				point meta min=-0.02, point meta max=0.82,
				colormap name=time_per_it_ratio_colors,
				colorbar horizontal, colorbar/width=1.75mm,
				colorbar style={
					at={(0.5,0)}, anchor=south, width=0.6\linewidth,
					title={Per-iteration time ratio (log scale)},
					title style={font=\footnotesize, yshift=-1.5mm},
					tick style={black, opacity=0.2, line width=0.75pt},
					tick label style={font=\footnotesize},
					xtick={-0.02, -0.02 + 0.84/6, -0.02 + 0.84*2/6, -0.02 + 0.84*3/6, -0.02 + 0.84*4/6, -0.02 + 0.84*5/6, 0.82},
					xticklabel={
						\pgfmathparse{10^\tick}
						\pgfmathprintnumber[fixed, fixed zerofill, precision=2]{\pgfmathresult}
					}
				}
			]
		\end{axis}
		        
		\path (group c1r3.north) -- (group c2r3.north) coordinate[midway] (mid3);
		\begin{axis}[
				at={(mid3)}, yshift=0.6cm, anchor=south,
				hide axis, scale only axis, height=0pt, width=0.5\linewidth,
				point meta min=-0.783, point meta max=0.783,
				colormap name=speedup_colors,
				colorbar horizontal, colorbar/width=1.75mm,
				colorbar style={
					at={(0.5,0)}, anchor=south, width=0.6\linewidth,
					xmin=-0.783, xmax=0.783,
					clip=false,
					xtick={-0.783, -0.783*2/3, -0.783/3, 0, 0.783/3, 0.783*2/3, 0.783},
					title={Speedup (log scale)},
					title style={font=\footnotesize, yshift=-1.5mm},
					tick style={black, opacity=0.2, line width=0.75pt},
					tick label style={font=\footnotesize},
					xticklabel={
						\pgfmathparse{10^\tick}
						\pgfmathprintnumber[fixed, fixed zerofill, precision=2]{\pgfmathresult}
					},
					after end axis/.code={
						\node[anchor=east, xshift=-3.5mm] at (rel axis cs:0, 0.5) {
							\begingroup
							\pgfsetfillcolor{rgb,255:red,0;green,51;blue,110}%
							\pgfsetstrokecolor{rgb,255:red,0;green,24;blue,51}%
							\pgfsetlinewidth{0.4pt}%
							\pgfsetplotmarksize{1mm}%
							\pgfuseplotmark{*}%
							\endgroup
						};
						\node[anchor=west, xshift=3.5mm] at (rel axis cs:1, 0.5) {
							\begingroup
							\pgfsetfillcolor{rgb,255:red,122;green,35;blue,0}%
							\pgfsetstrokecolor{rgb,255:red,61;green,17;blue,0}%
							\pgfsetlinewidth{0.4pt}%
							\pgfsetplotmarksize{0.875mm}%
							\pgfuseplotmark{square*}%
							\endgroup
						};
					}
				}
			]
		\end{axis}
		
		\node[anchor=north, font=\footnotesize, yshift=-22pt] at (group c1r1.south) {(a) $k=4$};
		\node[anchor=north, font=\footnotesize, yshift=-22pt] at (group c2r1.south) {(b) $k=8$};
		\node[anchor=north, font=\footnotesize, yshift=-22pt] at (group c1r2.south) {(c) $k=4$};
		\node[anchor=north, font=\footnotesize, yshift=-22pt] at (group c2r2.south) {(d) $k=8$};
		\node[anchor=north, font=\footnotesize, yshift=-26pt] at (group c1r3.south) {(e) $k=4$};
		\node[anchor=north, font=\footnotesize, yshift=-26pt] at (group c2r3.south) {(f) $k=8$};
	\end{tikzpicture}
	\setlength{\abovecaptionskip}{5pt}
	\captionsetup{font=footnotesize}
	\caption{Performance metrics of MBPCG relative to CG at $k=4$ and $k=8$. These plots visualize the linear relationships in \cref{tab:regression-it,tab:regression-time-per-it,tab:regression-speedup-2-input}. For iteration ratio and speedup, $k=8$ has greater gains than $k=4$ in the best cases but greater penalties in the worst cases. $k=8$ almost always incurs higher per-iteration execution times, especially when $n$ is large. Red indicates high MBPCG effectiveness, and blue indicates low effectiveness.}
	\label{fig:scatter}
\end{figure}

\section{Summary}
The MBPCG method's effectiveness depends heavily on properties of the system. In some systems, it significantly reduces iteration count and solve time. In others, it slightly increases iteration count and significantly increases solve time. Our statistical model analyzes trends behind these metrics and predicts whether the method will succeed. We find that speedup is positively associated with CG solve time and the number of nonzeros, while it is negatively associated with $n$. Improved SpMM kernels or hardware capabilities can increase the effectiveness of this method if they lead to lower asymptotic solve times with respect to $k$. There is also room to accelerate convergence by using better initial guesses, which do not necessarily need to be random. The idea of solving multiple copies of a system in a BCG-like approach can be applied to other CG variants or other iterative methods to potentially achieve speedups.

\section*{Acknowledgments}
We thank Slaven Pele\v{s} for writing the proposal that funded this work.

\bibliographystyle{siamplain}
\bibliography{bibfile}

\begin{landscape}
	\appendix
	\renewcommand{\thesection}{}
	\section*{Appendix}
	\label[appendix]{sec:Appendix}
	\addcontentsline{toc}{section}{Appendix}
	    
	\centering
	\setlength{\tabcolsep}{3.5pt}
	    
	\pgfplotstableset{
		matrix col style/.style={
			column name=Matrix, 
			string type, 
			column type={@{\extracolsep{\fill}}l}, 
			string replace*={_}{\_},
			postproc cell content/.append style={
				/pgfplots/table/@cell content/.add={\ttfamily}{}%
			}
		}
	}
	
	\pgfplotstableread[col sep=comma]{data/experiment_results.csv}\datatable
	\pgfplotstableforeachcolumnelement{Best k}\of\datatable\as\bestkcell{%
		\expandafter\xdef\csname best_k@\pgfplotstablerow\endcsname{\bestkcell}%
	}
	 
	\vspace{1em}
	\captionof{table}{Test matrices and experimental data for sparse systems, $n \leq 4{,}515$. Bold indicates the best solve strategy.}
	\label{tab:sparse-results-1}
	\pgfplotstabletypeset[
		begin table={\begin{tabular*}{\linewidth}},
		end table={\end{tabular*}},
		col sep=comma,
		font=\footnotesize,
		skip rows between index={29}{1000},
		every head row/.style={
			before row={
				\toprule
				\multicolumn{3}{c}{} & \multicolumn{4}{c}{\textbf{Iteration count}} & \multicolumn{4}{c}{\textbf{Total solve time (ms)}} & \multicolumn{4}{c}{\textbf{Time per iteration (ms)}} \\
				\cmidrule(lr){4-7} \cmidrule(lr){8-11} \cmidrule(lr){12-15}
			},
			after row=\midrule
		},
		every first row/.style={
			before row={\rule{0pt}{3ex}} 
		},
		every last row/.style={after row=\bottomrule},
		columns={Matrix, n, nnz, {Regular CG It}, {k=2 It}, {k=4 It}, {k=8 It}, {Regular CG Time}, {k=2 Time}, {k=4 Time}, {k=8 Time}, {Regular CG It Time}, {k=2 It Time}, {k=4 It Time}, {k=8 It Time}},
		columns/Matrix/.style={matrix col style},
		columns/n/.style={column name=$n$, fixed, precision=0, column type={r}},
		columns/nnz/.style={column name=nnz, fixed, precision=0, column type={r}},
		columns/{Regular CG It}/.style={column name={CG}, fixed, precision=0, column type={r}},
		columns/{k=2 It}/.style={column name={$k=2$}, fixed, precision=0, column type={r}},
		columns/{k=4 It}/.style={column name={$k=4$}, fixed, precision=0, column type={r}},
		columns/{k=8 It}/.style={column name={$k=8$}, fixed, precision=0, column type={r}},
		columns/{Regular CG Time}/.style={
			column name={CG}, fixed, fixed zerofill, precision=2, column type={r},
			postproc cell content/.append code={
				\edef\bestkvalue{\csname best_k@\pgfplotstablerow\endcsname}
				\ifdefstring{\bestkvalue}{CG}{
					\pgfkeysgetvalue{/pgfplots/table/@cell content}\besttmp
					\edef\besttmp{{\noexpand\boldmath\besttmp}}
					\pgfkeyslet{/pgfplots/table/@cell content}\besttmp
					}{}
			}
		},
		columns/{k=2 Time}/.style={
			column name={$k=2$}, fixed, fixed zerofill, precision=2, column type={r},
			postproc cell content/.append code={
				\edef\bestkvalue{\csname best_k@\pgfplotstablerow\endcsname}
				\ifdefstring{\bestkvalue}{k=2}{
					\pgfkeysgetvalue{/pgfplots/table/@cell content}\besttmp
					\edef\besttmp{{\noexpand\boldmath\besttmp}}
					\pgfkeyslet{/pgfplots/table/@cell content}\besttmp
					}{}
			}
		},
		columns/{k=4 Time}/.style={
			column name={$k=4$}, fixed, fixed zerofill, precision=2, column type={r},
			postproc cell content/.append code={
				\edef\bestkvalue{\csname best_k@\pgfplotstablerow\endcsname}
				\ifdefstring{\bestkvalue}{k=4}{
					\pgfkeysgetvalue{/pgfplots/table/@cell content}\besttmp
					\edef\besttmp{{\noexpand\boldmath\besttmp}}
					\pgfkeyslet{/pgfplots/table/@cell content}\besttmp
					}{}
			}
		},
		columns/{k=8 Time}/.style={
			column name={$k=8$}, fixed, fixed zerofill, precision=2, column type={r},
			postproc cell content/.append code={
				\edef\bestkvalue{\csname best_k@\pgfplotstablerow\endcsname}
				\ifdefstring{\bestkvalue}{k=8}{
					\pgfkeysgetvalue{/pgfplots/table/@cell content}\besttmp
					\edef\besttmp{{\noexpand\boldmath\besttmp}}
					\pgfkeyslet{/pgfplots/table/@cell content}\besttmp
					}{}
			}
		},
		columns/{Regular CG It Time}/.style={column name={CG}, fixed, fixed zerofill, precision=3, column type={r}},
		columns/{k=2 It Time}/.style={column name={$k=2$}, fixed, fixed zerofill, precision=3, column type={r}},
		columns/{k=4 It Time}/.style={column name={$k=4$}, fixed, fixed zerofill, precision=3, column type={r}},
		columns/{k=8 It Time}/.style={column name={$k=8$}, fixed, fixed zerofill, precision=3, column type={r}}
	]\datatable
	 
	\clearpage
	 
	\captionof{table}{Test matrices and experimental data for sparse systems, $4{,}800 \leq n \leq 40{,}000$. Bold indicates the best solve strategy.}
	\label{tab:sparse-results-2}
	\pgfplotstabletypeset[
		begin table={\begin{tabular*}{\linewidth}},
		end table={\end{tabular*}},
		col sep=comma,
		font=\footnotesize,
		skip rows between index={0}{29},
		skip rows between index={59}{1000},
		every head row/.style={
			before row={
				\toprule
				\multicolumn{3}{c}{} & \multicolumn{4}{c}{\textbf{Iteration count}} & \multicolumn{4}{c}{\textbf{Total solve time (ms)}} & \multicolumn{4}{c}{\textbf{Time per iteration (ms)}} \\
				\cmidrule(lr){4-7} \cmidrule(lr){8-11} \cmidrule(lr){12-15}
			},
			after row=\midrule
		},
		every first row/.style={
			before row={\rule{0pt}{3ex}} 
		},
		every last row/.style={after row=\bottomrule},
		columns={Matrix, n, nnz, {Regular CG It}, {k=2 It}, {k=4 It}, {k=8 It}, {Regular CG Time}, {k=2 Time}, {k=4 Time}, {k=8 Time}, {Regular CG It Time}, {k=2 It Time}, {k=4 It Time}, {k=8 It Time}},
		columns/Matrix/.style={matrix col style},
		columns/n/.style={column name=$n$, fixed, precision=0, column type={r}},
		columns/nnz/.style={column name=nnz, fixed, precision=0, column type={r}},
		columns/{Regular CG It}/.style={column name={CG}, fixed, precision=0, column type={r}},
		columns/{k=2 It}/.style={column name={$k=2$}, fixed, precision=0, column type={r}},
		columns/{k=4 It}/.style={column name={$k=4$}, fixed, precision=0, column type={r}},
		columns/{k=8 It}/.style={column name={$k=8$}, fixed, precision=0, column type={r}},
		columns/{Regular CG Time}/.style={
			column name={CG}, fixed, fixed zerofill, precision=2, column type={r},
			postproc cell content/.append code={
				\edef\bestkvalue{\csname best_k@\pgfplotstablerow\endcsname}
				\ifdefstring{\bestkvalue}{CG}{
					\pgfkeysgetvalue{/pgfplots/table/@cell content}\besttmp
					\edef\besttmp{{\noexpand\boldmath\besttmp}}
					\pgfkeyslet{/pgfplots/table/@cell content}\besttmp
					}{}
			}
		},
		columns/{k=2 Time}/.style={
			column name={$k=2$}, fixed, fixed zerofill, precision=2, column type={r},
			postproc cell content/.append code={
				\edef\bestkvalue{\csname best_k@\pgfplotstablerow\endcsname}
				\ifdefstring{\bestkvalue}{k=2}{
					\pgfkeysgetvalue{/pgfplots/table/@cell content}\besttmp
					\edef\besttmp{{\noexpand\boldmath\besttmp}}
					\pgfkeyslet{/pgfplots/table/@cell content}\besttmp
					}{}
			}
		},
		columns/{k=4 Time}/.style={
			column name={$k=4$}, fixed, fixed zerofill, precision=2, column type={r},
			postproc cell content/.append code={
				\edef\bestkvalue{\csname best_k@\pgfplotstablerow\endcsname}
				\ifdefstring{\bestkvalue}{k=4}{
					\pgfkeysgetvalue{/pgfplots/table/@cell content}\besttmp
					\edef\besttmp{{\noexpand\boldmath\besttmp}}
					\pgfkeyslet{/pgfplots/table/@cell content}\besttmp
					}{}
			}
		},
		columns/{k=8 Time}/.style={
			column name={$k=8$}, fixed, fixed zerofill, precision=2, column type={r},
			postproc cell content/.append code={
				\edef\bestkvalue{\csname best_k@\pgfplotstablerow\endcsname}
				\ifdefstring{\bestkvalue}{k=8}{
					\pgfkeysgetvalue{/pgfplots/table/@cell content}\besttmp
					\edef\besttmp{{\noexpand\boldmath\besttmp}}
					\pgfkeyslet{/pgfplots/table/@cell content}\besttmp
					}{}
			}
		},
		columns/{Regular CG It Time}/.style={column name={CG}, fixed, fixed zerofill, precision=3, column type={r}},
		columns/{k=2 It Time}/.style={column name={$k=2$}, fixed, fixed zerofill, precision=3, column type={r}},
		columns/{k=4 It Time}/.style={column name={$k=4$}, fixed, fixed zerofill, precision=3, column type={r}},
		columns/{k=8 It Time}/.style={column name={$k=8$}, fixed, fixed zerofill, precision=3, column type={r}}
	]\datatable
	 
	\clearpage
	 
	\captionof{table}{Test matrices and experimental data for sparse systems, $40{,}806 \leq n \leq 220{,}542$. Bold indicates the best solve strategy.}
	\label{tab:sparse-results-3}
	\pgfplotstabletypeset[
		begin table={\begin{tabular*}{\linewidth}},
		end table={\end{tabular*}},
		col sep=comma,
		font=\footnotesize,
		skip rows between index={0}{59},
		skip rows between index={89}{1000},
		every head row/.style={
			before row={
				\toprule
				\multicolumn{3}{c}{} & \multicolumn{4}{c}{\textbf{Iteration count}} & \multicolumn{4}{c}{\textbf{Total solve time (ms)}} & \multicolumn{4}{c}{\textbf{Time per iteration (ms)}} \\
				\cmidrule(lr){4-7} \cmidrule(lr){8-11} \cmidrule(lr){12-15}
			},
			after row=\midrule
		},
		every first row/.style={
			before row={\rule{0pt}{3ex}} 
		},
		every last row/.style={after row=\bottomrule},
		columns={Matrix, n, nnz, {Regular CG It}, {k=2 It}, {k=4 It}, {k=8 It}, {Regular CG Time}, {k=2 Time}, {k=4 Time}, {k=8 Time}, {Regular CG It Time}, {k=2 It Time}, {k=4 It Time}, {k=8 It Time}},
		columns/Matrix/.style={matrix col style},
		columns/n/.style={column name=$n$, fixed, precision=0, column type={r}},
		columns/nnz/.style={column name=nnz, fixed, precision=0, column type={r}},
		columns/{Regular CG It}/.style={column name={CG}, fixed, precision=0, column type={r}},
		columns/{k=2 It}/.style={column name={$k=2$}, fixed, precision=0, column type={r}},
		columns/{k=4 It}/.style={column name={$k=4$}, fixed, precision=0, column type={r}},
		columns/{k=8 It}/.style={column name={$k=8$}, fixed, precision=0, column type={r}},
		columns/{Regular CG Time}/.style={
			column name={CG}, fixed, fixed zerofill, precision=2, column type={r},
			postproc cell content/.append code={
				\edef\bestkvalue{\csname best_k@\pgfplotstablerow\endcsname}
				\ifdefstring{\bestkvalue}{CG}{
					\pgfkeysgetvalue{/pgfplots/table/@cell content}\besttmp
					\edef\besttmp{{\noexpand\boldmath\besttmp}}
					\pgfkeyslet{/pgfplots/table/@cell content}\besttmp
					}{}
			}
		},
		columns/{k=2 Time}/.style={
			column name={$k=2$}, fixed, fixed zerofill, precision=2, column type={r},
			postproc cell content/.append code={
				\edef\bestkvalue{\csname best_k@\pgfplotstablerow\endcsname}
				\ifdefstring{\bestkvalue}{k=2}{
					\pgfkeysgetvalue{/pgfplots/table/@cell content}\besttmp
					\edef\besttmp{{\noexpand\boldmath\besttmp}}
					\pgfkeyslet{/pgfplots/table/@cell content}\besttmp
					}{}
			}
		},
		columns/{k=4 Time}/.style={
			column name={$k=4$}, fixed, fixed zerofill, precision=2, column type={r},
			postproc cell content/.append code={
				\edef\bestkvalue{\csname best_k@\pgfplotstablerow\endcsname}
				\ifdefstring{\bestkvalue}{k=4}{
					\pgfkeysgetvalue{/pgfplots/table/@cell content}\besttmp
					\edef\besttmp{{\noexpand\boldmath\besttmp}}
					\pgfkeyslet{/pgfplots/table/@cell content}\besttmp
					}{}
			}
		},
		columns/{k=8 Time}/.style={
			column name={$k=8$}, fixed, fixed zerofill, precision=2, column type={r},
			postproc cell content/.append code={
				\edef\bestkvalue{\csname best_k@\pgfplotstablerow\endcsname}
				\ifdefstring{\bestkvalue}{k=8}{
					\pgfkeysgetvalue{/pgfplots/table/@cell content}\besttmp
					\edef\besttmp{{\noexpand\boldmath\besttmp}}
					\pgfkeyslet{/pgfplots/table/@cell content}\besttmp
					}{}
			}
		},
		columns/{Regular CG It Time}/.style={column name={CG}, fixed, fixed zerofill, precision=3, column type={r}},
		columns/{k=2 It Time}/.style={column name={$k=2$}, fixed, fixed zerofill, precision=3, column type={r}},
		columns/{k=4 It Time}/.style={column name={$k=4$}, fixed, fixed zerofill, precision=3, column type={r}},
		columns/{k=8 It Time}/.style={column name={$k=8$}, fixed, fixed zerofill, precision=3, column type={r}}
	]\datatable
	 
	\clearpage
	 
	\captionof{table}{Test matrices and experimental data for sparse systems, $245{,}874 \leq n$. Bold indicates the best solve strategy.}
	\label{tab:sparse-results-4}
	\pgfplotstabletypeset[
		begin table={\begin{tabular*}{\linewidth}},
		end table={\end{tabular*}},
		col sep=comma,
		font=\footnotesize,
		skip rows between index={0}{89},
		skip rows between index={107}{1000},
		every head row/.style={
			before row={
				\toprule
				\multicolumn{3}{c}{} & \multicolumn{4}{c}{\textbf{Iteration count}} & \multicolumn{4}{c}{\textbf{Total solve time (ms)}} & \multicolumn{4}{c}{\textbf{Time per iteration (ms)}} \\
				\cmidrule(lr){4-7} \cmidrule(lr){8-11} \cmidrule(lr){12-15}
			},
			after row=\midrule
		},
		every first row/.style={
			before row={\rule{0pt}{3ex}} 
		},
		every last row/.style={after row=\bottomrule},
		columns={Matrix, n, nnz, {Regular CG It}, {k=2 It}, {k=4 It}, {k=8 It}, {Regular CG Time}, {k=2 Time}, {k=4 Time}, {k=8 Time}, {Regular CG It Time}, {k=2 It Time}, {k=4 It Time}, {k=8 It Time}},
		columns/Matrix/.style={matrix col style},
		columns/n/.style={column name=$n$, fixed, precision=0, column type={r}},
		columns/nnz/.style={column name=nnz, fixed, precision=0, column type={r}},
		columns/{Regular CG It}/.style={column name={CG}, fixed, precision=0, column type={r}},
		columns/{k=2 It}/.style={column name={$k=2$}, fixed, precision=0, column type={r}},
		columns/{k=4 It}/.style={column name={$k=4$}, fixed, precision=0, column type={r}},
		columns/{k=8 It}/.style={column name={$k=8$}, fixed, precision=0, column type={r}},
		columns/{Regular CG Time}/.style={
			column name={CG}, fixed, fixed zerofill, precision=2, column type={r},
			postproc cell content/.append code={
				\edef\bestkvalue{\csname best_k@\pgfplotstablerow\endcsname}
				\ifdefstring{\bestkvalue}{CG}{
					\pgfkeysgetvalue{/pgfplots/table/@cell content}\besttmp
					\edef\besttmp{{\noexpand\boldmath\besttmp}}
					\pgfkeyslet{/pgfplots/table/@cell content}\besttmp
					}{}
			}
		},
		columns/{k=2 Time}/.style={
			column name={$k=2$}, fixed, fixed zerofill, precision=2, column type={r},
			postproc cell content/.append code={
				\edef\bestkvalue{\csname best_k@\pgfplotstablerow\endcsname}
				\ifdefstring{\bestkvalue}{k=2}{
					\pgfkeysgetvalue{/pgfplots/table/@cell content}\besttmp
					\edef\besttmp{{\noexpand\boldmath\besttmp}}
					\pgfkeyslet{/pgfplots/table/@cell content}\besttmp
					}{}
			}
		},
		columns/{k=4 Time}/.style={
			column name={$k=4$}, fixed, fixed zerofill, precision=2, column type={r},
			postproc cell content/.append code={
				\edef\bestkvalue{\csname best_k@\pgfplotstablerow\endcsname}
				\ifdefstring{\bestkvalue}{k=4}{
					\pgfkeysgetvalue{/pgfplots/table/@cell content}\besttmp
					\edef\besttmp{{\noexpand\boldmath\besttmp}}
					\pgfkeyslet{/pgfplots/table/@cell content}\besttmp
					}{}
			}
		},
		columns/{k=8 Time}/.style={
			column name={$k=8$}, fixed, fixed zerofill, precision=2, column type={r},
			postproc cell content/.append code={
				\edef\bestkvalue{\csname best_k@\pgfplotstablerow\endcsname}
				\ifdefstring{\bestkvalue}{k=8}{
					\pgfkeysgetvalue{/pgfplots/table/@cell content}\besttmp
					\edef\besttmp{{\noexpand\boldmath\besttmp}}
					\pgfkeyslet{/pgfplots/table/@cell content}\besttmp
					}{}
			}
		},
		columns/{Regular CG It Time}/.style={column name={CG}, fixed, fixed zerofill, precision=3, column type={r}},
		columns/{k=2 It Time}/.style={column name={$k=2$}, fixed, fixed zerofill, precision=3, column type={r}},
		columns/{k=4 It Time}/.style={column name={$k=4$}, fixed, fixed zerofill, precision=3, column type={r}},
		columns/{k=8 It Time}/.style={column name={$k=8$}, fixed, fixed zerofill, precision=3, column type={r}}
	]\datatable
	 
	\clearpage
	 
	\captionof{table}{Test matrices and experimental data for dense systems. Bold indicates the best solve strategy.}
	\label{tab:dense-results}
	\pgfplotstabletypeset[
		begin table={\begin{tabular*}{\linewidth}},
		end table={\end{tabular*}},
		col sep=comma,
		font=\footnotesize,
		skip rows between index={0}{107},
		every head row/.style={
			before row={
				\toprule
				\multicolumn{2}{c}{} & \multicolumn{4}{c}{\textbf{Iteration count}} & \multicolumn{4}{c}{\textbf{Total solve time (ms)}} & \multicolumn{4}{c}{\textbf{Time per iteration (ms)}} \\
				\cmidrule(lr){3-6} \cmidrule(lr){7-10} \cmidrule(lr){11-14}
			},
			after row=\midrule
		},
		every first row/.style={
			before row={\rule{0pt}{3ex}} 
		},
		every last row/.style={after row=\bottomrule},
		columns={Matrix, n, {Regular CG It}, {k=2 It}, {k=4 It}, {k=8 It}, {Regular CG Time}, {k=2 Time}, {k=4 Time}, {k=8 Time}, {Regular CG It Time}, {k=2 It Time}, {k=4 It Time}, {k=8 It Time}},
		columns/Matrix/.style={column name=Matrix, string type, column type={@{\extracolsep{\fill}}l}, string replace*={_}{\_}},
		columns/n/.style={column name=$n$, fixed, precision=0, column type={r}},
		columns/{Regular CG It}/.style={column name={CG}, fixed, precision=0, column type={r}},
		columns/{k=2 It}/.style={column name={$k=2$}, fixed, precision=0, column type={r}},
		columns/{k=4 It}/.style={column name={$k=4$}, fixed, precision=0, column type={r}},
		columns/{k=8 It}/.style={column name={$k=8$}, fixed, precision=0, column type={r}},
		columns/{Regular CG Time}/.style={
			column name={CG}, fixed, fixed zerofill, precision=2, column type={r},
			postproc cell content/.append code={
				\edef\bestkvalue{\csname best_k@\pgfplotstablerow\endcsname}
				\ifdefstring{\bestkvalue}{CG}{
					\pgfkeysgetvalue{/pgfplots/table/@cell content}\besttmp
					\edef\besttmp{{\noexpand\boldmath\besttmp}}
					\pgfkeyslet{/pgfplots/table/@cell content}\besttmp
					}{}
			}
		},
		columns/{k=2 Time}/.style={
			column name={$k=2$}, fixed, fixed zerofill, precision=2, column type={r},
			postproc cell content/.append code={
				\edef\bestkvalue{\csname best_k@\pgfplotstablerow\endcsname}
				\ifdefstring{\bestkvalue}{k=2}{
					\pgfkeysgetvalue{/pgfplots/table/@cell content}\besttmp
					\edef\besttmp{{\noexpand\boldmath\besttmp}}
					\pgfkeyslet{/pgfplots/table/@cell content}\besttmp
					}{}
			}
		},
		columns/{k=4 Time}/.style={
			column name={$k=4$}, fixed, fixed zerofill, precision=2, column type={r},
			postproc cell content/.append code={
				\edef\bestkvalue{\csname best_k@\pgfplotstablerow\endcsname}
				\ifdefstring{\bestkvalue}{k=4}{
					\pgfkeysgetvalue{/pgfplots/table/@cell content}\besttmp
					\edef\besttmp{{\noexpand\boldmath\besttmp}}
					\pgfkeyslet{/pgfplots/table/@cell content}\besttmp
					}{}
			}
		},
		columns/{k=8 Time}/.style={
			column name={$k=8$}, fixed, fixed zerofill, precision=2, column type={r},
			postproc cell content/.append code={
				\edef\bestkvalue{\csname best_k@\pgfplotstablerow\endcsname}
				\ifdefstring{\bestkvalue}{k=8}{
					\pgfkeysgetvalue{/pgfplots/table/@cell content}\besttmp
					\edef\besttmp{{\noexpand\boldmath\besttmp}}
					\pgfkeyslet{/pgfplots/table/@cell content}\besttmp
					}{}
			}
		},
		columns/{Regular CG It Time}/.style={column name={CG}, fixed, fixed zerofill, precision=3, column type={r}},
		columns/{k=2 It Time}/.style={column name={$k=2$}, fixed, fixed zerofill, precision=3, column type={r}},
		columns/{k=4 It Time}/.style={column name={$k=4$}, fixed, fixed zerofill, precision=3, column type={r}},
		columns/{k=8 It Time}/.style={column name={$k=8$}, fixed, fixed zerofill, precision=3, column type={r}}
	]\datatable
\end{landscape}
 
\end{document}